\documentclass{amsart}
\usepackage{graphicx}
\usepackage{comment}
\usepackage{color}
\newtheorem{thm}{Theorem}[section]
\newtheorem{cor}[thm]{Corollary}
\newtheorem{lem}[thm]{Lemma}
\newtheorem{prop}[thm]{Proposition}

\theoremstyle{definition}

\theoremstyle{remark}
\newtheorem{rem}[thm]{Remark}
\numberwithin{equation}{section}
\newcommand{\norm}[1]{\left\Vert#1\right\Vert}

\newcommand{\R}{\mathbb R}
\newcommand{\C}{\mathbb C}
\DeclareMathOperator{\Li}{Li}

\renewcommand{\Re}{\operatorname{Re}}
\renewcommand{\Im}{\operatorname{Im}}

\begin{document}


\title[A fractional derivative approximation on bounded domains] 
{A higher order resolvent-positive finite difference approximation for fractional derivatives on bounded domains} 


\author{Boris Baeumer} 
\address[Boris Baeumer and Matthew Parry]{Department of Mathematics \& Statistics, University of Otago, New Zealand}
\email{bbaeumer@maths.otago.ac.nz} 
\author{Mih\'aly Kov\'acs}
\address[Mih\'aly Kov\'acs]{Faculty of Information Technology and Bionics, P\'azm\'any P\'eter Catholic University,
Budapest, Hungary}
\address[Mih\'aly Kov\'acs]{Department of Differential Equations,
Budapest University of Technology and Economics,
Budapest,  Hungary}
\email{kovacs.mihaly@itk.ppke.hu}

\author{Matthew Parry}  \email{mparry@maths.otago.ac.nz}







\date{\today}%

\begin{abstract}We develop a finite difference approximation of order $\alpha$ for the $\alpha$-fractional derivative. The weights of the approximation scheme have the same rate-matrix type properties as the popular Gr\"unwald scheme. In particular, approximate solutions to fractional diffusion equations preserve positivity. Furthermore, for the approximation of the solution to the skewed fractional heat equation on a bounded domain the new approximation scheme keeps its order $\alpha$ whereas the order of the Gr\"unwald scheme reduces to order $\alpha-1$, contradicting the convergence rate results by Meerschaert and Tadjeran. 
\end{abstract}





\maketitle 

\section{Introduction}

In \cite{Baeumer2017c,Baeumer2018} it was shown that 
for any  $1<\alpha\le 2$, the Cauchy problem
\begin{equation}\label{FPDE}u'(t,x)=\begin{cases}\left(\frac{d}{dx}\right)^\alpha u(t,x)&x\in(0,1)\\0&\mbox{ else}\end{cases},
\end{equation} 
with initial condition $u(0,x)=u_0(x)$ is well-posed on $L_1(0,1)$ and $C_0(0,1)$ and the solution was approximated with a backward Euler (time)/Gr\"unwald (space) scheme. It was observed that as soon as the solution has mass near the left boundary the apparent order of the spatial approximation scheme reduced to order $\alpha-1$ due to the fact that at that stage the eigenfunction 
\[u_c(x)=\sum_{n=1}^\infty \frac{c^{n-1}x^{n\alpha-1}}{\Gamma(n\alpha)}\]
(with $c<0$ largest such that $\sum_{n=1}^\infty \frac{c^{n-1}}{\Gamma(n\alpha)}=0$) becomes dominant and
behaves like $x^{\alpha-1} $ near the left boundary. The fact that there are $\alpha$-times differentiable functions in $C_0(0,1)$ whose zero extension to the real line is not $\alpha$-times differentiable was overlooked in  \cite{Meerschaert2004d,Tadjeran2006} while arguing for first-order convergence of the Gr\"unwald scheme. 
\begin{figure}[htb]

\centering
\includegraphics[scale=.5]{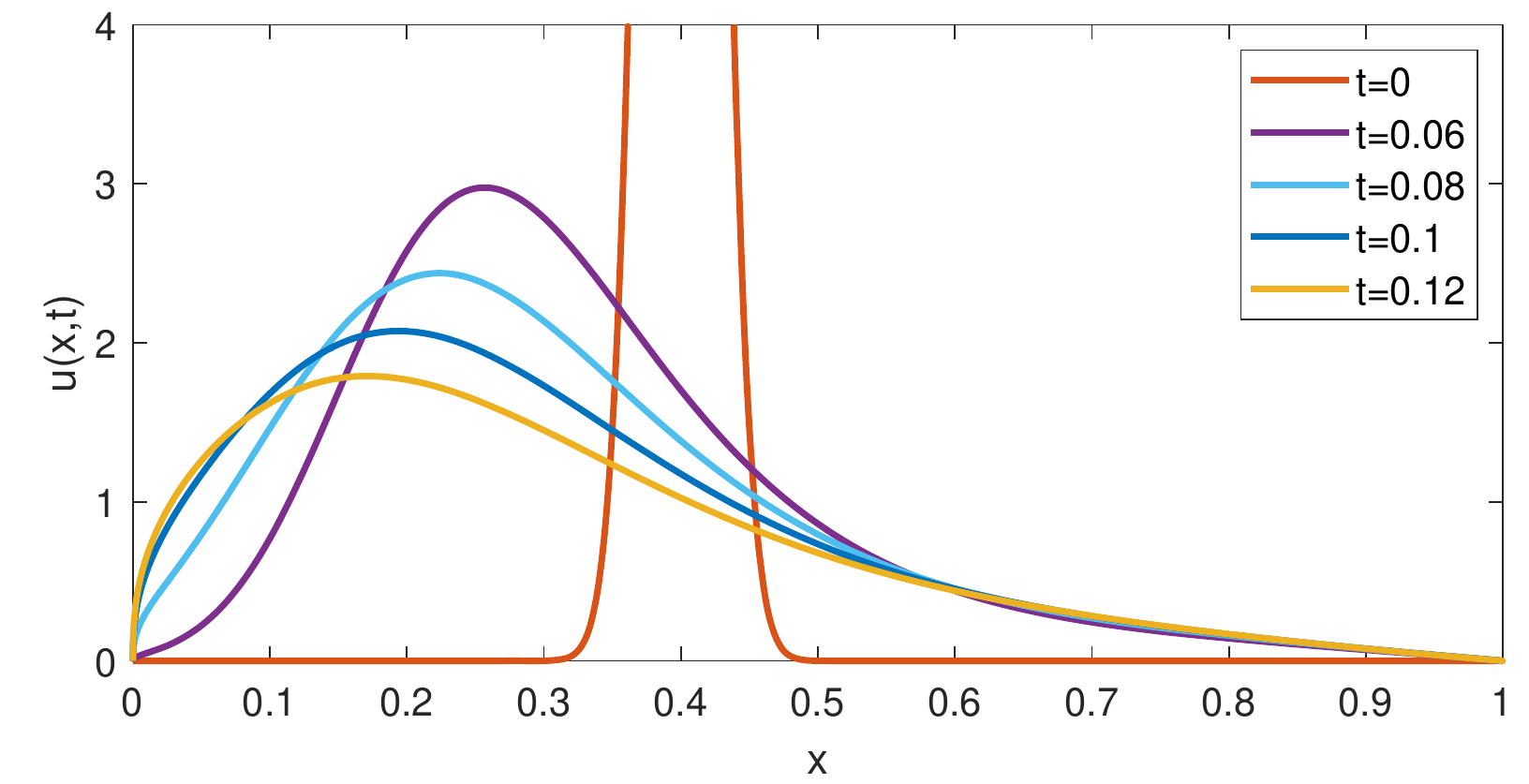}
\includegraphics[scale=.65]{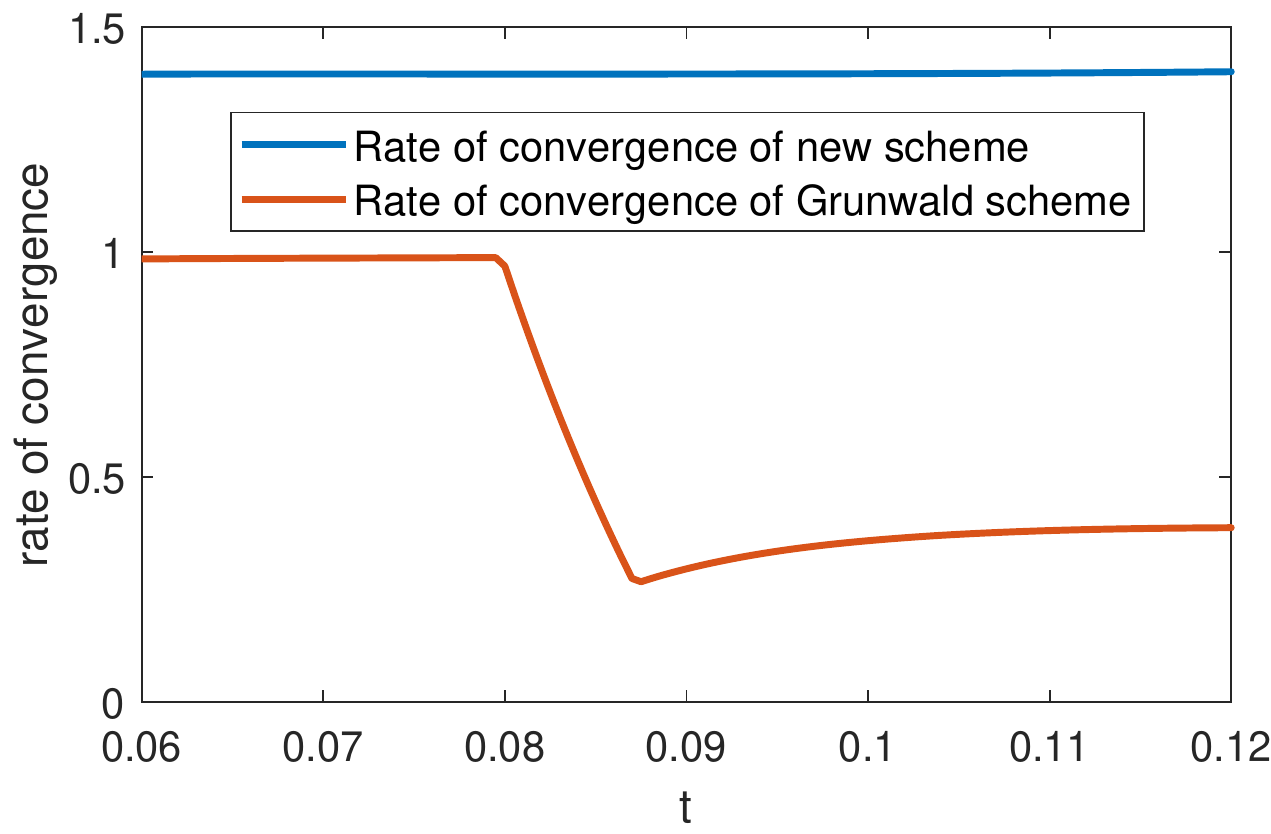}
\caption{Plot of approximate solutions and observed error rate to \eqref{FPDE} for $\alpha=1.4$ and $u_0(x)=\exp(-(x-0.4)^2/2\sigma^2)/\sqrt{2\pi\sigma^2}$ with $\sigma^2=0.0005$. The relative error for $n=400$ is about 0.3\% for the new scheme and 2\% for the Grunwald scheme.}
\end{figure}
In order to deal with this boundary behaviour we developed a new spatial scheme
\begin{equation}\label{scheme}A_h^\alpha f(x)=\frac1{h^\alpha}\sum_{n=-1}^\infty w_{n+1} f(x-nh)\end{equation}
that turns out to be of order $\alpha$ for all $2\alpha$-times differentiable functions as well as for the spatially semidiscrete approximation of \eqref{FPDE}. 

The paper is organized as follows.

In Section 2 we introduce the scheme, state the main results, and prove basic properties of the weights $w_n$. 

In Section 3 we first consider the problem on $\mathbb{R}$ and  show that  the scheme $A_h^\alpha f\to \left(\frac{d}{dx}\right)^\alpha f$ converges with order $\beta\le\alpha$ on all spaces where the shift group is strongly continuous  and $f$ is $\alpha+\beta$-times differentiable. We furthermore show that $A^\alpha_h$ generate uniformly analytic semigroups.

In Section 4 we prove that the scheme provides convergence of order $\alpha$ to solutions to \eqref{FPDE}. We use a piece-wise power interpolation, which is exact for $x^{\alpha-1}$, to obtain this rate of convergence for all $x$, not just on grid points.

\section{The scheme}
In order to avoid convergence-destroying residuals from approximating the fractional derivative of $x^{\alpha-1}/\Gamma(\alpha)$ (the $\delta$-function) we are looking for weights that make the scheme \eqref{scheme}  \emph{exact} for $x^{\alpha-1}$; i.e. we are looking for weights such that \eqref{scheme} 
 gives exactly $1/h$ at node $0$ if applied to $x^{\alpha-1}/\Gamma(\alpha)$. In particular,
\begin{equation}\label{newMatrix}
  \frac{1}{h^\alpha}\begin{pmatrix} 
       w_1 & w_0 & 0 & \ldots & 0    \\
    w_2 & w_1&    \ddots &\ddots & \vdots \\
    \vdots&\ddots&\ddots    & \ddots &0    \\
    \vdots&\ddots&\ddots&w_1&w_0  \\
    w_n &\cdots&\cdots&w_2&w_1\\
  \end{pmatrix}
  \begin{pmatrix}0\\h^{\alpha-1}\\(2h)^{\alpha-1}\\\vdots\\\vdots\\((n-1)h)^{\alpha-1}\end{pmatrix}
  =\begin{pmatrix}
  \Gamma(\alpha)/h\\0\\\vdots\\\\\vdots\\0
  \end{pmatrix}.
\end{equation}
Multiplying both sides by $h$ and shifting the vector one up (deleting the first column of the matrix) we obtain the easily and quickly solvable tri-diagonal system determining the weights $w_j$,
\begin{equation}\label{newMatrixgen}
  \begin{pmatrix} 
       w_0 &0 & \ldots & 0    \\
    w_1 & w_0&    \ddots &\vdots \\
    \vdots&\ddots&\ddots    & 0    \\
    w_{n-1}&\cdots&w_1&w_0  \\
    \end{pmatrix}
  \begin{pmatrix}1\\2^{\alpha-1}\\\vdots\\n^{\alpha-1}\end{pmatrix}
  =\begin{pmatrix}
  \Gamma(\alpha)\\0\\\vdots\\0
  \end{pmatrix}.
\end{equation}

Our two main results regarding the scheme \eqref{scheme} are as follows.
\begin{thm}\label{thm:convonR}
Let $X=C_0(\R)$ or $X=L^p(\R)$, $1\le p<\infty$. Let $A=d/dx$ and $1<\alpha\le 2$ and $0<\beta\le \alpha$.  Then there exists $C>0$ such that for $f\in D(A^{\alpha+\beta})$,
\[\|A^\alpha f-A^\alpha_h f\|\le Ch^\beta \|A^{\alpha+\beta}f\|,\]
where $A^\alpha_h f$ is the scheme defined in \eqref{scheme} using the weights satisfying \eqref{newMatrixgen}.
\end{thm}
\begin{proof}
See Section \ref{infdomain}.
\end{proof}
We denote with $C_0(\Omega)$ the completion with respect to the supremum norm of the space of  continuous functions that are compactly supported within $\Omega$; i.e. $C(0,1)$ is the space of continuous functions $f$ that satisfy $f(0)=f(1)=0$ and $C(0,1]$ is the space of continuous functions $f$ satisfying $f(0)=0$.

On a bounded domain we have the following result regarding the solution semigroup  $\{S(t)\}_{t\ge 0}$ to \eqref{FPDE} on $C_0(0,1)$ and the matrix semigroup $\{\mathbf{S}_h(t)\}_{t\ge 0}$ solving its finite difference approximation 
\[\frac{d}{dt} \vec{u}_h(t)=M_h\vec{u}_h(t); \vec{u}_h(0)=\vec{u}_0,\]
with $h=1/(n+1)$, $\vec u_h=(u_h^1,u_h^2,\ldots,u_h^n)$, and $\mathbf{M}_h$ being the matrix on the left of \eqref{newMatrix}; that is,
\begin{equation}\label{eq:mn}
   \mathbf{M}_h:= \frac{1}{h^\alpha}\begin{pmatrix} 
       w_1 & w_0 & 0 & \ldots & 0    \\
    w_2 & w_1&    \ddots &\ddots & \vdots \\
    \vdots&\ddots&\ddots    & \ddots &0    \\
    \vdots&\ddots&\ddots&w_1&w_0  \\
    w_n &\cdots&\cdots&w_2&w_1\\
  \end{pmatrix}.
\end{equation}
\begin{thm}
Let $u_0\in C_0(0,1)$ be such that
\[u_0(x)=ax^{\alpha-1}+bx^{2\alpha-1}+\int_0^x\frac{(x-s)^{\alpha}}{\Gamma(\alpha+1)}g(s)\,ds\] for some $g\in C_0(0,1]$. Then there exists $C>0$ such that for $t>h^\alpha>0$,
\[ \left|[S(t)u_0](x_i)-\left[\mathbf{S}_h(t)\vec{u}_0\right]_i\right|\le Ch^\alpha(1+\log(t/h^\alpha))(\|u_0\|_\infty+|b|+\|g\|_\infty), \]
for all $1\le i\le n$, where $\vec{u}_0=(u_0(x_1),\dots,u_0(x_n))$.
\end{thm}
\begin{proof}
See Section \ref{bddomain}.
\end{proof}

The fact that $\textbf{M}_h$ and $A_h^\alpha$ generate positive contraction semigroups follows from the following proposition showing that $M_h$ is a `$Q$-matrix'. 

\begin{prop}\label{conjecture1} Consider the weights $w_j$  on the left hand side of \eqref{newMatrixgen}. Then
$w_1<0$, $w_j>0$ for all $j\neq1$ and $\sum_{j=0}^\infty w_j=0$.
 \end{prop}
\begin{proof}
Taking the Laplace transform of the convolution in the infinite version of \eqref{newMatrixgen} yields
$$\left(\sum_{k=0}^\infty w_k e^{-\lambda  k}\right) \left(e^{\lambda }\sum_{j=1}^\infty j^{\alpha-1}e^{-j\lambda}\right)=\Gamma(\alpha)$$
or 
\begin{align}
    \sum_{k=0}^\infty w_k e^{-\lambda  k}&=\Gamma(\alpha)\Bigg/ \left(e^{\lambda} \sum_{j=1}^\infty j^{\alpha-1}e^{-j\lambda}\right)\nonumber\\
    &=\frac{e^{-\lambda }\Gamma(\alpha)}{\Li_{1-\alpha}(e^{-\lambda })},\label{polylog}
\end{align}
 with
\[\Li_{s}(z):=\sum_{j=1}^\infty j^{-s}z^j\]
being the \emph{Polylogarithm} function. 
It is well known that the polylogarithm
for fixed $s$ is analytic in $\C\setminus[1,\infty)$ and is real on $\R\setminus[1,\infty)$. 
It has the expansion  for $|\mu|<2\pi$ and $s\neq 1,2,\ldots$ given by
\begin{equation}\label{zetaexpansion}\Li_{s}(e^{-\mu})=\Gamma(1-s)\mu^{s-1}+\sum_{j=0}^\infty \frac{\zeta(s-j)(-1)^j}{j!}\mu^j,\end{equation}
where $\zeta$ is the Riemann Zeta function (see e.g. \cite[Eq. (13)]{truesdell1945function}, also in Equation (9.3) in \cite{Wood1992}). For large moduli; i.e., for $r\to\infty$ and for $-\pi<\theta<\pi$, if $s\neq -1,-2,\ldots$  (\cite[Eq. (11.3)]{Wood1992})
 \begin{equation}\label{polyasym} \Li_{s}(-re^{i\theta})=-\frac{(i\theta+\ln r)^{s}}{\Gamma(s+1)}+O((\ln r)^{s-2} ).\end{equation}
Furthermore, along the branch cut for $x>1$ (\cite[Eq. (3.1)]{Wood1992}), 
\begin{equation}\label{branch}\lim_{\epsilon\to 0}\Im(\Li_{1-\alpha}(x\pm i\epsilon))=\pm \frac{\pi (\ln x)^{-\alpha}}{\Gamma(1-\alpha)}.\end{equation}
Finally,
for $-1<s<0$, $\Li_s$ has a single root at $0$  by
 the Argument Principle along the keyhole contour $\gamma$ of $\Li$. This is due to the fact that  $\Li_{1-\alpha}(\gamma)\cap\R^+=\Li_{1-\alpha}(\gamma\cap\R^+)$ which follows from the asymptotic behaviours of the polylogarithmic function near the branch cut and near infinity. This in turn implies that the winding number of $\Li_{1-\alpha}(\gamma)$ is one (see also \cite[Page 1]{10.1307/mmj/1029003073} and \cite[Theorem 4]{peyerimhoff1966zeros}).

Letting $t=e^{-\lambda}$, equation \eqref{polylog} turns into
\[\sum_{k=0}^\infty w_k t^k=\frac{t\Gamma(\alpha)}{\sum_{j=1}^\infty j^{\alpha-1} t^j}=\frac{t\Gamma(\alpha)}{\Li_{1-\alpha}(t)},\]
which is analytic in $\C\setminus[1,\infty)$ as $\Li_{1-\alpha}$ has only a single root at zero.
Hence by the Residue Theorem
\[w_n=\frac1{2\pi i}\int_\gamma \frac{\Gamma(\alpha)}{z^n\Li_{1-\alpha}(z)}\,dz,\] where $\gamma$ is any closed counter-clockwise simple curve around the origin.
We use a keyhole contour for $n\ge 2$, evading $[1,\infty)$. By \eqref{polyasym} and the expansion for small $z$, the integrals over the outer and inner arc converge to zero. By \eqref{branch} this leaves
\begin{align*}
w_n=&\lim_{\epsilon\to0}\frac{\Gamma(\alpha)}{2\pi i}\int_1^\infty \frac1{x^n\Li_{1-\alpha}(x+i\epsilon)}- \frac1{x^n\Li_{1-\alpha}(x-i\epsilon)}\,dx\\
=&\frac{\Gamma(\alpha)}{2\pi i}\int_1^\infty \frac1{x^n(\Re (\Li_{1-\alpha}(x))+i\pi(\ln x)^{-\alpha}/\Gamma(1-\alpha))}\\&\;\;\;\;\;\;\;\;\;\;\;\;- \frac1{x^n(\Re (\Li_{1-\alpha}(x))-i\pi(\ln x)^{-\alpha}/\Gamma(1-\alpha))}\,dx\\
=& -\frac{\Gamma(\alpha)}{\Gamma(1-\alpha)}\int_1^\infty \frac{(\ln x)^{-\alpha}}{x^n|\Li_{1-\alpha}(x)|^2}\,dx>0.
\end{align*}

Finally, taking $t\to 1$ (or $\lambda\to 0$), 
\[\sum_{k=0}^\infty w_k =
\lim_{t\to 1}\frac{t\Gamma(\alpha)}{\Li_{1-\alpha}(t)}=0,\]
and the proof is complete.
 \end{proof}

\section{Convergence on an infinite domain}\label{infdomain}
We  follow \cite{Baeumer2012a} and first show convergence of the scheme on $L^1(\R)$. This will then imply convergence on all spaces where the shift-group is strongly continuous (even more general, the appropriately adapted scheme will converge to the $\alpha$ power of the generator of a strongly continuous group). 

In order to introduce some notation define the Fourier transform for a function in $L^1(\R)$ via 
\[\hat f(k)=\int e^{-ikx}f(x)\,dx.\]
Let
\[X_\alpha:=\{f\in L^1(\R):\exists f^{(\alpha)}\in L^1(\R):\widehat{ f^{(\alpha)}}(k)=(ik)^\alpha \hat f(k) \}\]
and for $f\in X_\alpha$ define
\[\norm f_\alpha:=\norm {f^{(\alpha)}}_{L^1(\R)}.\]

\begin{thm}\label{l1convergence}
Let $1<\alpha\le 2$ and $0\le\beta\le \alpha$. Then there exists $C>0$ such that for $f\in X_{\alpha+\beta}$ with   $A_h^\alpha f$ defined via \eqref{scheme}, 
\[ \norm{A_h^\alpha f- f^{(\alpha)}}_{L^1(\R)}\le Ch^\beta\norm{f}_{\alpha+\beta}.\]
In case of $\beta=0$ and $f\in X_\alpha$, $\norm{A_h^\alpha f- f^{(\alpha)}}_{L^1(\R)}\to 0$.
\end{thm}
\begin{proof}

The Fourier transform of our approximation scheme is given by
\begin{align*}
    \widehat{A_h^\alpha f}(k)&=\frac1{h^\alpha}\sum_{n=-1}^\infty w_{n+1} e^{-iknh}\hat f(k)=\frac{\Gamma(\alpha)}{h^\alpha\Li_{1-\alpha}(e^{-ihk})}\hat f(k)\\
    &=(ik)^\alpha \hat g(kh)\hat f(k)\\
    &=(ik)^\alpha \hat f(k)+h^\beta\frac{\hat g(kh)-1}{(ikh)^\beta}(ik)^{\alpha+\beta}\hat f(k)
\end{align*}
with \[\hat g(k)
=\frac{\Gamma(\alpha)}{(ik)^\alpha\Li_{1-\alpha}(e^{-ik})}.\]
We claim that 
\begin{equation}\label{gbeta}\hat g_\beta(k):=(\hat g(k)-1)/(ik)^\beta\end{equation}
is the Fourier transform of an $L^1$-function for all $0< \beta\le \alpha$. This would prove the the theorem for all $0< \beta\le \alpha$ as the $L^1$ norm of $g_\beta$ is the same as the $L^1$-norm of $x\mapsto g_\beta(x/h)/h$.   But first we show that $\hat g$ is the Fourier transform of an $L^1$ function with $\hat g(0)=1$, which obviously implies the assertion of the theorem in case of $\beta=0$.

We are going to use a variant of the Carlson-Beurling Inequality (see \cite[Prop. 2.1]{Baeumer2012a})
\begin{equation}\label{Carlson}\| f\|_{L^1}\le C(r)\|\hat f\|_{L^r}^{\frac1s}\|\hat f'\|_{L^r}^{\frac1r},\end{equation}
where $1<r\le 2$ and $1/r+1/s=1$.
Recall that by \eqref{zetaexpansion} 
\[\Li_s(e^{-ik})=(ik)^{s-1}\Gamma(1-s)+\zeta(s)+O(k),\] 
for $|k|\le\pi$. Hence $\hat g(0)=1 $ and as $|\Li_{1-\alpha}(e^{-ik})|$ is bounded below, $\hat g\in L^2(\R)$. Furthermore,
\begin{align}\label{li1}\frac{d}{dk}\hat g(k)=&-i\frac{\alpha\Gamma(\alpha)}{(ik)^{\alpha+1}\Li_{1-\alpha}(e^{-ik})}-\frac{\Gamma(\alpha)}{(ik)^\alpha\Li_{1-\alpha}^2(e^{-ik})}(-i\Li_{-\alpha}(e^{-ik}))\\\nonumber
=&-i\hat g\left(e^{-ik}\right)\left(\frac{\alpha}{ik}-\frac{\Li_{-\alpha}(e^{-ik})}{\Li_{1-\alpha}(e^{-ik})}\right)\\\nonumber
=&-i\hat g\left(e^{-ik}\right)\frac{\alpha\Li_{1-\alpha}(e^{-ik})-ik\Li_{-\alpha}(e^{-ik})}{ik\Li_{1-\alpha}(e^{-ik})}\\
=&-i\hat g\left(e^{-ik}\right)\frac{\alpha(ik)^{\alpha-1}\Li_{1-\alpha}(e^{-ik})-(ik)^\alpha\Li_{-\alpha}(e^{-ik})}{(ik)^\alpha\Li_{1-\alpha}(e^{-ik})}=O(k^{\alpha-1}).\label{li2}
\end{align}
Note that in \eqref{li1}, the periodic function $\Li_{-\alpha}(e^{-ik})/\Li^2_{1-\alpha}(e^{-ik})$ is bounded as at the singularity $\lim_{k\to 0}\Li_{-\alpha}(e^{-ik})/\Li^2_{1-\alpha}(e^{-ik})=\lim_{k\to 0}(ik)^{-\alpha-1}/(ik)^{-2\alpha}=0$.  Equation \eqref{li2} shows that $d\hat g/dk(0)=0$ and combined with the decay at infinity given by \eqref{li1}, $d\hat g/dk\in L^2(\R)$ as well. Hence $g\in L^1(\R)$.

In order to show that $g_\beta$ is in $L^1$, note that for $|k|<\pi$,
\[\hat g_\beta(k)=\frac{\hat g(k)-1}{(ik)^\beta}=(ik)^{\alpha-\beta} (\zeta(1-\alpha)+O(k))/\Gamma(\alpha),\]
and that for large $|k|$, $|\hat g_\beta(k)|$ decays like $|k|^{-\beta}$; i.e., there exists $M>0$ such that \[\left|\hat g_\beta(k)\right|\le M|k|^{-\beta}.\]
As to its derivative,
\[\frac{d}{dk}\frac{\hat g(k)-1}{(ik)^\beta}=i\frac{k\hat g'(k)-\beta(\hat g(k)-1)}{(ik)^{\beta+1}}.
\]
So for large $|k|$ (i.e. $|k|\ge 1$) there exists $M$ such that 
\[\left|\hat g'_\beta(k)\right|\le M|k|^{-\beta-1}.\]  However for small $k$ and $\beta<\alpha$,  $|\hat g'_\beta (k)|\approx|k|^{\alpha-\beta-1}$, whereas \eqref{zetaexpansion} implies that for $\beta=\alpha$ the function is bounded in zero. 

Consider the partition of unity ($j=1,2,\ldots$)
\begin{align*}\phi_0(k):=& \begin{cases}1&|k|<1\\2-|k|&1\le|k|\le2\\0& \mbox{else}\end{cases},\\
\phi_j(k):=&\begin{cases}(|k|-2^{j-1})2^{1-j}&2^{j-1}\le |k|\le 2^j\\(2^{j+1}-|k|)2^{-j}&2^j\le |k|\le 2^{j+1}\\0&\mbox{else}\end{cases}.
\end{align*}
Then by \eqref{Carlson}, $\phi_j\hat g_\beta$ are the Fourier transforms of $L^1$ functions and for $j\ge 1$, using the tail estimates of $\hat g_\beta$ and $\hat g'_\beta$,
\[\norm{\phi_j\hat g_\beta}_2^{1/2}\norm{(\phi_j\hat g_\beta)'}_2^{1/2}\le C2^{-\beta j}.\]
As $\hat g_\beta=\sum_{j=0}^\infty \phi_j \hat g_\beta$, $g_\beta\in L^1(\R)$.

\end{proof}

\subsection{Transference to other Banach spaces}
In order to transfer Theorem \ref{l1convergence} we use the unbounded functional calculus developed in \cite{Baeumer2009a}. In particular, recall that for $\hat f(k)=\int e^{-ikt} f(t)\,dt$ and $A$ being the generator of a strongly continuous group $(G(t))_{t\in\R}$ on a Banach space $X$, one may define
\[ \hat f(-iA):=\int_{\mathbb{R}} f(t)G(t)\,dt,\]
where the integral is understood in the strong Bochner sense.

\begin{thm}
Let $-A$ be the generator of a bounded strongly continuous group $(G(t))_{t\in\R}$ on a Banach space $X$ and define
\[ A_h^\alpha f:=\frac1{h^\alpha}\sum_{n=-1}^\infty w_{n+1} G(nh)f.\]
Then for $0\le\beta<\alpha$ and $f\in D(A^{\alpha+\beta})$,
\[\norm{A^\alpha_h f-A^\alpha f}_X\le Ch^\beta \norm{A^{\alpha+\beta}f}_X.\]
In case $\beta=0$, $\norm{A^\alpha_h f-A^\alpha f}_X\to 0$.
\end{thm}

\begin{proof}
Consider $\hat g_\beta$ of \eqref{gbeta}. Then 
\[\norm{h^\beta \hat g_\beta(-ihA)}_{\mathcal B(X)}\le Ch^\beta\]
and hence $\norm{h^\beta \hat g_\beta(-ihA)A^{\alpha+\beta}f}_X\le Ch^\beta\norm{A^{\alpha+\beta}f}_X $. But
\[ h^\beta \hat g_\beta(-ihA)A^{\alpha+\beta}f=A_h^\alpha f-A^\alpha f.\]
In case of $\beta=0$, the strong continuity of $G$ implies that \[\hat g_0(-ihA)f=\int_{\mathbb{\R}} G(s)fg(s/h)/h\,ds-f\to 0\] for all $f\in X$.
\end{proof}

\begin{cor}
The approximation scheme converges uniformly at rate $h^\alpha$ in $C_0(\R)$ for $2\alpha$-times continuously differentiable functions.
\end{cor}

\subsection{Uniform analyticity of approximating semigroups}
Following \cite[Section 4]{Baeumer2012a}, let 
\[\psi(k):=\frac{\Gamma(\alpha)}{\Li_{1-\alpha}(e^{-ik})}=\sum_{j=-1}^\infty w_{j+1} e^{-ikj}.\]
Using that $\sum w_j=0$, $w_1<0$ and $w_j>0$ for all $j\neq 1$, we deduce that $\Re\psi(k)<0$ for $k\neq 0$. As \[\frac{d}{dk}\psi(k)=i\frac{\Gamma(\alpha)\Li_{-\alpha}(e^{-ik})}{(\Li_{1-\alpha}(e^{-ik}))^2}\approx \frac{\Gamma(\alpha)(ik)^{-\alpha-1}}{(ik)^{-2\alpha}\Gamma(\alpha+1)}=\frac1\alpha(ik)^{\alpha-1},\] 
the conditions of \cite[Theorem 4.1]{Baeumer2012a} are satisfied.

This shows that the Fourier multipliers defined by our approximation generate uniformly analytic semigroups of angle $\alpha$ on $L_1(\mathbb{R})$. Following \cite[Section 4]{Baeumer2012a}, this implies if one employs the backward Euler scheme as time integrator, then the fully discrete scheme converges unconditionally at rate $\Delta t+(\Delta x)^\alpha\log\Delta x$ to the solution of the Cauchy problem on $\R$ for initial conditions $u_0\in D(A^\alpha)$ in any space where the translation group is strongly continuous, in particular in $L_1(\mathbb{R})$ and $C_0(\mathbb{R})$.

\section{Dirichlet problem on bounded domain}\label{bddomain}

In order to prove convergence of order $\alpha$ in space to solutions of \eqref{FPDE}, we first show that the solution semigroup is analytic  and that the matrix $M_h$ governing the approximation scheme can be used to define  generators of  contraction semigroups on finite dimensional subspaces. We then show that the semigroups converge to each other at the required rate if the initial condition is smooth enough even though in general  the generators converge at a much slower rate. 

To keep notation consistent we denote the generators in this section also by $A^\alpha$ and $A_h^\alpha$, but we would like to stress that due to the imposition of boundary conditions these operators are not fractional powers of the first derivative operator anymore.

\subsection{Analyticity of solution semigroup}
Recall from \cite{Baeumer2017c} that the solution semigroup governing \eqref{FPDE} on $C_0(0,1)$ has generator $(A^\alpha,D(A^\alpha))$ with
\begin{equation}\label{domain}D(A^\alpha):=\{f:f(x)=I^\alpha g(x)-I^\alpha g(1)x^{\alpha-1},\, g\in C_0(0,1)\}\end{equation}
and $A^\alpha f=g$ for $f\in D(A^\alpha )$; here $I^\alpha g(x)=\int_0^x\frac{(x-s)^{\alpha-1}}{\Gamma(\alpha)}g(s)\,ds$ is the $\alpha$-fractional antiderivative of $g$.

 In a weak sense, the solution to
\[\lambda f-A^\alpha f=\delta_y\] 
is given for $\Re\lambda\ge 0$ by 
\begin{align*}f_y(x):=&-H(x-y)\sum_{n=1}^\infty \frac{(x-y)^{n\alpha-1}}{\Gamma(n\alpha)}\lambda^{n-1}\\
&+\frac{\sum_{n=1}^\infty \frac{(1-y)^{n\alpha-1}}{\Gamma(n\alpha)}\lambda^{n-1}}{\sum_{n=1}^\infty \frac{1}{\Gamma(n\alpha)}\lambda^{n-1}}\sum_{n=1}^\infty \frac{x^{n\alpha-1}}{\Gamma(n\alpha)}\lambda^{n-1}.\end{align*}
Hence
\[R(\lambda,A^\alpha )g(x)=\int_0^1 g(y)f_y(x)\,dy.\]
The function $f_y$ can be expressed in terms of the Mittag-Leffler function
\[E_{\alpha,0}(z)=\sum_{n=1}^\infty \frac{z^n}{\Gamma(n\alpha)}\] via
\begin{align*}
    f_y=&-H(x-y)\frac1{\lambda(x-y)}E_{\alpha,0}(\lambda (x-y)^\alpha)\\&+\frac1{\lambda(1-y)x}\frac{E_{\alpha,0}(\lambda (1-y)^\alpha)}{E_{\alpha,0}(\lambda)}E_{\alpha,0}(\lambda x^\alpha).
\end{align*}
For $|z|$ large with $\arg z<\alpha\pi/2$ it has the asymptotic expansion (see  \cite[\S18]{Bateman1953})
\[E_{\alpha,0}(z)= \frac{z^{1/\alpha}}{\alpha}\exp(z^{1/\alpha})-\frac1{z\Gamma(-\alpha)}+O(|z|^{-2}).\]
To show that $\lambda R(\lambda,A^\alpha )$ is bounded we start with an estimate for the Green's function $f_y$.
\begin{prop}\label{propanalytic}
Let $|\lambda|$ be large enough so that $|E_{\alpha_,0}(\lambda)|\ge \left|\frac{\lambda^{1/\alpha}}{2\alpha}\exp(\lambda^{1/\alpha})\right|$. Then there exists $M>0$ such that for all $x,y\in (0, 1)$,
\[|f_y(x)|\le \begin{cases} M \left|\frac{\lambda^{1/\alpha}}{\lambda}\right|\exp(-(y-x)\Re\lambda^{1/\alpha})& x<y+|\lambda|^{-1/\alpha}\\
M /(|\lambda|^2(x-y)^{1+\alpha})& \mbox{else}\end{cases}.\]
\end{prop}
\begin{proof}
Note that
\begin{align*}|f_y(x)|\le & \left|\frac{2\alpha\exp(-\lambda^{1/\alpha})}{\lambda^{1+1/\alpha}}\right|\times\\&\left|\frac{E_{\alpha,0}(\lambda(1-y)^\alpha)}{1-y}\frac{E_{\alpha,0}(\lambda x^\alpha)}{x}-H(x-y)\frac{E_{\alpha,0}(\lambda(x-y)^\alpha)}{x-y}E_{\alpha,0}(\lambda)\right|.
\end{align*}
Furthermore, for $|\lambda| r^\alpha<1$, there exists $M>0$ such that 
\[\left|\frac{E_{\alpha,0}(\lambda r^\alpha)}{r}\right|\le M|\lambda|r^{\alpha-1}\le M|\lambda|^{1/\alpha}\] 
and for $|\lambda| r^\alpha\ge1$,
\[\left|\frac{E_{\alpha,0}(\lambda r^\alpha)}{r}\right|=\left|\lambda^{1/\alpha}
\exp(\lambda^{1/\alpha}r)/\alpha+h(\lambda r^\alpha)/r\right|\le |\lambda^{1/\alpha}|\left(|\exp(\lambda^{1/\alpha}r)|+M\right);\]
with $h$ being the tail of the asymptotic expansion satisfying $|h(z)|\le M/|z|$. 
In both cases we obtain
\begin{equation}
    \label{MLestimate}
    \left|\frac{E_{\alpha,0}(\lambda r^\alpha)}{r}\right|\le M|\lambda|^{1/\alpha}\exp(r\Re \lambda^{1/\alpha}),
\end{equation}
proving the estimate for $|f_y(x)|$ in case of $x<y+|\lambda|^{-1/\alpha}$. 

In case of $x\ge y+|\lambda|^{-1/\alpha}$, 
\begin{align*}|f_y(x)|\le & \left|\frac{2\alpha\exp(-\lambda^{1/\alpha})}{\lambda^{1+1/\alpha}}\right|\times\\& \left|\frac{h(\lambda(1-y)^\alpha)}{1-y}\frac{\lambda^{1/\alpha}}{\alpha}\exp(\lambda^{1/\alpha}x)+\frac{h(\lambda x^\alpha)}{x}\frac{\lambda^{1/\alpha}}{\alpha}\exp(\lambda^{1/\alpha}(1-y))\right.\\
&\left.+\frac{h(\lambda(1-y)^\alpha)}{1-y}\frac{h(\lambda x^\alpha)}{x}-\frac{h(\lambda(x-y)^\alpha)}{x-y}\frac{\lambda^{1/\alpha}}{\alpha}\exp(\lambda^{1/\alpha})\right.\\
&-\left.h(\lambda ) \frac{\lambda^{1/\alpha}}{\alpha}\exp(\lambda^{1/\alpha}(x-y))-\frac{h(\lambda(x-y)^\alpha)}{x-y}h(\lambda )\right|\\
\le & \frac{M}{|\lambda^2|(x-y)^{1+\alpha}}.
\end{align*}
\end{proof}

\begin{thm}
Let $(A^\alpha ,D(A^\alpha ))$ be the generator of the semigroup given by the fractional derivative operator of order $1<\alpha\le 2$ on $C_0(0,1)$ or $L^1[0,1]$ with Dirichlet boundary conditions. Then there exists $M>0$ such that for all $\Re\lambda\ge 0$,
\[\norm{ R(\lambda,A^\alpha )}\le M/|\lambda|;\] i.e., the semigroup is analytic.
\end{thm}
\begin{proof}
Note that from above,
\[ R(\lambda, A^\alpha ) g(x)=\int_0^1 f_y(x) g(y)\,dy.\]
Hence, by Proposition \ref{propanalytic}, if the underlying Banach space is $X=C_0(0,1)$, 
\begin{align*}\norm{R(\lambda,A^\alpha )}=&\sup_{\norm g=1}\norm{R(\lambda, A^\alpha )g}=\sup_{\norm g=1}\sup_{x\in(0,1)}\left|\int_0^1 f_y(x)g(y)\,dy\right|\\
\le&\sup_{x\in(0,1)} \int_0^1 \left|f_y(x)\right|\,dy\\
=&\sup_{x\in(0,1)}\int_0^{x-|\lambda|^{-1/\alpha}\vee 0}|f_y(x)|\,dy+\int_{x-|\lambda|^{-1/\alpha}\vee 0}^1|f_y(x)|\,dy\\
\le &\sup_{x\in(0,1)}\frac{M}{|\lambda|^2}\int_0^{x-|\lambda|^{-1/\alpha}\vee 0}\frac1{(x-y)^{1+\alpha}}\,dy\\
&+\frac{M|\lambda|^{1/\alpha}}{|\lambda|}\int_{x-|\lambda|^{-1/\alpha}\vee0}^1\exp(-(y-x)\Re\lambda^{1/\alpha})\,dy\\
\le&\frac{M}{\alpha|\lambda|^2}|\lambda|+\frac{M}{|\lambda|}\frac{|\lambda|^{1/\alpha}}{\Re\lambda^{1/\alpha}}\exp\left(|\lambda|^{-1/\alpha}\Re\lambda^{1/\alpha}\right)\\
=&\frac{M}{\alpha|\lambda|}+\frac{M}{|\lambda|} \frac{\exp(\cos(\pi/2\alpha))}{\cos(\pi/2\alpha)}.
\end{align*}
Similarly for $X=L^1[0,1]$, one has
\[\norm{R(\lambda,A^\alpha )}\le\sup_{y\in[0,1]}\int_0^1 |f_y(x)|\,dx\le M/|\lambda|.\] 
\end{proof}

\subsection{Piece-wise power interpolation}
As the rate of convergence hinges on mapping the function $x^{\alpha-1}$ exactly, we introduce a \emph{piece-wise power interpolation.}
For an $n$-dimensional vector $y=(y_1,\ldots,y_n)$ define
\[(E_n y)(x):= \frac{y_{i+1}-y_i}{x_{i+1}^{\alpha-1}-x_i^{\alpha-1}}x^{\alpha-1}+\frac{x_{i+1}^{\alpha-1}y_i-x_i^{\alpha-1}y_{i+1}}{x_{i+1}^{\alpha-1}-x_i^{\alpha-1}},\]
where $ih=x_i\le x\le x_{i+1}=(i+1)h$ and $(n+1)h=1$ and $y_0=y_{n+1}=0$.

Let 
$\Pi_n:C(0,1)\to C(0,1)$ be the projection to the $n$-dimensional power interpolation subspace $V_n$ via
\[\Pi_n f:=E_n(f(x_i)_{i=1,\dots,n}),\]
where \[V_n=\{\Pi_nf:f\in C_0(0,1)\}\subset C_0(0,1).\]
Note that $\Pi_n f(x_i)=f(x_i)$ for all $x_i=ih$, $i=0,\ldots,n+1$. Furthermore, for continuous functions in $C_0(0,1]$ we extend the projection in the obvious way. In particular, $\Pi_n(ax^{\alpha-1})=ax^{\alpha-1}$.

Define \[A^\alpha _hf:= E_n\mathbf{M}_h(f(x_i))_{i=1,\dots,n},\] where $\mathbf{M}_h$ is given by \eqref{eq:mn}.
By mapping the values on the grid points $x_i, 1\le i\le n$ and the properties of the weights given in Proposition \ref{conjecture1}; i.e. $\mathbf{M}_h$ is a so-called $Q$-matrix, the operator $A^\alpha _h$ is the generator of a contraction semigroup $\{S_h(t)\}_{t\geq 0}$ on the $n$-dimensional power interpolation subspace $V_n$ and is a bounded linear operator on $C_0(0,1)$. For $f\in C_0(0,1]$ we extend $A^\alpha _h$ by increasing the dimension of $\mathbf{M}_h$ by one and setting the last row equal to zero. In particular $A^\alpha _h x^{\alpha-1}=0$.

\begin{prop}\label{projectionrate}
Let $(A^\alpha ,D(A^\alpha ))$ be the generator of the semigroup given by the fractional derivative operator of order $1<\alpha\le 2$ on $C_0(0,1)$ and
let $f\in D(A^\alpha )$ and $(n+1)h=1$. Then there exists $C>0$ such that 
\[\|\Pi_n f-f\|\le Ch^\alpha\|A^\alpha f\|.\]
\end{prop}

\begin{proof}
Let $f\in D(A^\alpha )$ and $x=x_i+\lambda h$ for some $x_i=ih$  and $0\le\lambda\le 1$. Then by \eqref{domain}, $f=f_0-f_0(1)x^{\alpha-1}$ with $f_0(x)=\int_0^x \frac{(x-s)^{\alpha-1}}{\Gamma(\alpha)}g(s)\,ds$ for some $g\in C_0(0,1)$. As the power interpolation is exact for $x^{\alpha-1}$, it follows that
\begin{align*}
    \Pi_nf(x)-f(x)=& \frac{x^{\alpha-1}-x_i^{\alpha-1}}{(x_i+h)^{\alpha-1}-x_i^{\alpha-1}}f_0(x_i+h)+\frac{(x_i+h)^{\alpha-1}-x^{\alpha-1}}{(x_i+h)^{\alpha-1}-x_i^{\alpha-1}}f_0(x_i)\\&-f_0(x)\\
    =&\frac{x^{\alpha-1}-x_i^{\alpha-1}}{(x_i+h)^{\alpha-1}-x_i^{\alpha-1}}(f_0(x_i+h)-f_0(x))\\
    &+\frac{(x_i+h)^{\alpha-1}-x^{\alpha-1}}{(x_i+h)^{\alpha-1}-x_i^{\alpha-1}}(f_0(x_i)-f_0(x)).
\end{align*}
Furthermore, 
\begin{align*}
   \left| f_0(x_i+h)\right.&\left.-f_0(x)-(1-\lambda)hf_0'(x)\right|\\=&\,
   \left|\int_0^{x_i+h}\frac{(x_i+h-s)^{\alpha-1}}{\Gamma(\alpha)}g(s)\,ds\right.-
    \int_0^{x}\frac{(x-s)^{\alpha-1}}{\Gamma(\alpha)}g(s)\,ds\\
    &\,\left.-(1-\lambda)h\int_0^{x}\frac{(x-s)^{\alpha-2}}{\Gamma(\alpha-1)}g(s)\,ds\right|\\
    =&\left|\int_0^{x}\frac{(x_i+h-s)^{\alpha-1}-(x-s)^{\alpha-1}-(\alpha-1)(1-\lambda)h(x-s)^{\alpha-2}}{\Gamma(\alpha)}g(s)\,ds\right.\\
    &\,\left.+\int_x^{x_i+h}\frac{(x_i+h-s)^{\alpha-1}}{\Gamma(\alpha)}g(s)\,ds\right|\\
    \le&\,\|g\|\left| \frac{(x_i+h-x)^\alpha-(x_i+h)^{\alpha}+x^\alpha+\alpha(1-\lambda) h x^{\alpha-1}}{\Gamma(\alpha+1)}\right|\\
    &+\|g\|\frac{(x_i+h-x)^\alpha}{\Gamma(\alpha+1)}\\
    \le&\,\|g\|\left(2\frac{(1-\lambda)^\alpha h^\alpha}{\Gamma(\alpha+1)}+\left|\frac{(x+(1-\lambda)h)^\alpha-x^\alpha-\alpha(1-\lambda)hx^{\alpha-1}}{\Gamma(\alpha+1)}\right|\right).
\end{align*}
For $x\ge h$ the right hand side is maximal at $x=h$ and hence
\[ \left| f_0(x_i+h)\right.\left.-f_0(x)-(1-\lambda)hf_0'(x)\right|\le Ch^\alpha.\]
Similarly,
\[\left| f_0(x)-f_0(x_i)-\lambda hf_0'(x_i)\right|\le Ch^\alpha.\]
Hence 
\begin{align*}
    \Pi_nf(x)-f(x)=& \frac{x^{\alpha-1}-x_i^{\alpha-1}}{(x_i+h)^{\alpha-1}-x_i^{\alpha-1}}(1-\lambda)hf_0'(x)\\
    &-\frac{(x_i+h)^{\alpha-1}-x^{\alpha-1}}{(x_i+h)^{\alpha-1}-x_i^{\alpha-1}}\lambda hf_0'(x_i)+O(h^\alpha).
\end{align*}
Observe that
\begin{align*}
    \frac{x^{\alpha-1}-x_i^{\alpha-1}}{(x_i+h)^{\alpha-1}-x_i^{\alpha-1}}
    =&\,\frac{(\alpha-1)\lambda hx_i^{\alpha-2}+O( h^2x_i^{\alpha-3})}{(x_i+h)^{\alpha-1}-x_i^{\alpha-1}}\\
    =&\, \frac{\lambda+O( h/x_i)}{(x_i+h)^{\alpha-1}-x_i^{\alpha-1})/((\alpha-1)hx_i^{\alpha-2})},
\end{align*}
and as the denominator is $1+O(h/x_i)$,
\[  \frac{(x_i+h)^{\alpha-1}-x^{\alpha-1}}{(x_i+h)^{\alpha-1}-x_i^{\alpha-1}}= \frac{1-\lambda+O( h/x_i)}{(x_i+h)^{\alpha-1}-x_i^{\alpha-1})/((\alpha-1)hx_i^{\alpha-2})}\]
and therefore
\begin{align*}
\Pi_nf(x)-f(x)=&\frac{\lambda(1-\lambda)h}{((x_i+h)^{\alpha-1}-x_i^{\alpha-1})/(\alpha-1)hx_i^{\alpha-2}}\left(f_0'(x)-f_0'(x_i)\right)\\
&+(1-\lambda) h O(h/x_i)f_0'(x)+ \lambda h O(h/x_i)f_0'(x_i).\end{align*}
Since $|f'_0(x)-f_0'(x_i)|\le \|g\|h^{\alpha-1}/\Gamma(\alpha)$ and $|f'_0(x)|\le \|g\|x^{\alpha-1}/\Gamma(\alpha)$, it follows that 
\begin{align*}
\left|\Pi_nf(x)-f(x)\right|\le &\frac{\lambda(1-\lambda)h}{(2^{\alpha-1}-1)/(\alpha-1)}h^{\alpha-1}\|g\|/\Gamma(\alpha)\\
&+Ch^2\left(\frac{(x_i+h)^{\alpha-1}}{x_i}+\frac{x_i^{\alpha-1}}{x_i}\right)\\
\le & Ch^\alpha\|g\|=Ch^\alpha\|A^\alpha f\|,\end{align*}
as decreasing $x_i$ increases the right hand side; i.e., the right hand side is maximal if $x_i=h$.

In case $x<h$, $x_i=0$ and hence
\[\left|\Pi_nf(x)-f(x)\right|=\frac{x^{\alpha-1}}{h^{\alpha-1}}|f_0(h)|\le \frac{h^\alpha}{\Gamma(\alpha+1)}\|g\|=\frac{h^\alpha}{\Gamma(\alpha+1)}\|A^\alpha f\|
\]
as well.
\end{proof}
\begin{rem}
Note that in the above proof the fact that $g(1)=0$ is not needed.
\end{rem}

\begin{prop}\label{inverseEst}
Let $g\in C_0(0,1)$. Assume there exists $C_g>0$ such that for all $n,j$ with $0\le j\le n-1$ and $h=1/(n+1), x_j=jh$ ,
\[\left|\frac{g(x_{j+1})-g(x_j)}{x_{j+1}^{\alpha-1}-x_j^{\alpha-1}}\right|<C_g.\]
Then there exists $C>0$ independent of $g$ such that
\[ \|\left(A^\alpha\right) ^{-1}\Pi_n g-\left(A^\alpha_h\right)^{-1}\Pi_n g\|\le Ch^\alpha( C_g+\|g\|).\]
\end{prop}
\begin{proof}
Note that for $\lambda=0$ the formula for the negative resolvent reads as 
\[\left(A^\alpha\right)^{-1}g(x)=\int_0^x \frac{(x-y)^{\alpha-1}}{\Gamma(\alpha)}g(y)\,dy-x^{\alpha-1}\int_0^1 \frac{(1-y)^{\alpha-1}}{\Gamma(\alpha)}g(y)\,dy,\]
which implies that 
\begin{align*}\left(A^\alpha\right)^{-1}\Pi_n& g(x_i)=\\
&\sum_{j=0}^{i-1}\int_{x_j}^{x_{j+1}}\frac{(x_i-y)^{\alpha-1}}{\Gamma(\alpha)}\left[\frac{y^{\alpha-1}-x_j^{\alpha-1}}{x_{j+1}^{\alpha-1}-x_j^{\alpha-1}}(g(x_{j+1})-g(x_j))+g(x_{j})\right]dy\\
-x_i^{\alpha-1}&\sum_{j=0}^{n}\int_{x_j}^{x_{j+1}}\frac{(1-y)^{\alpha-1}}{\Gamma(\alpha)}\left[\frac{y^{\alpha-1}-x_j^{\alpha-1}}{x_{j+1}^{\alpha-1}-x_j^{\alpha-1}}(g(x_{j+1})-g(x_j))+g(x_{j})\right]dy.
\end{align*}
Let $(n+1)h=1$, $x_j=jh$; $j=1,\ldots, n$, and $\textbf{M}_h$ be defined by  \eqref{eq:mn}. Then
\[\left(M_h^{-1}e_j\right)_i=h\left(H(i-j)\frac{(ih-jh)^{\alpha-1}}{\Gamma(\alpha)}-(hi)^{\alpha-1}\frac{(1-jh)^{\alpha-1}}{\Gamma(\alpha)}\right),\]
where $H$ is the Heaviside step function.
Hence 
\begin{align*}\left(A^\alpha _h\right)^{-1}\Pi_n g(x_i)=&\sum_{j=1}^i h\frac{(x_i-x_j)^{\alpha-1}}{\Gamma(\alpha)}g(x_j)-x_i^{\alpha-1}\sum_{j=1}^{n+1}h\frac{(1-x_j)^{\alpha-1}}{\Gamma(\alpha)}g(x_j)\\
=&\frac{h}2 \left(\sum_{j=0}^{i-1} \frac{(x_i-x_{j+1})^{\alpha-1}}{\Gamma(\alpha)}g(x_{j+1})+\sum_{j=0}^{i-1} \frac{(x_i-x_j)^{\alpha-1}}{\Gamma(\alpha)}g(x_{j})\right)\\
-&x_i^{\alpha-1}\frac h2\left( \sum_{j=0}^{n}\frac{(1-x_{j+1})^{\alpha-1}}{\Gamma(\alpha)}g(x_{j+1})+ \sum_{j=0}^{n}\frac{(1-x_{j})^{\alpha-1}}{\Gamma(\alpha)}g(x_{j})\right),
\end{align*}
using that $g(0)=0$. Note that $\left(A^\alpha _h\right)^{-1}\Pi_n g(x_i)$ is the trapezoidal rule approximation of $\left(A^\alpha \right)^{-1}\Pi_n g(x_i)$ and that the second integrals are a special case of the first integrals using $x_i=1$. 

Comparing each summand of $\left(A^\alpha\right) ^{-1}\Pi_ng(x_i)-\left(A^\alpha _h\right)^{-1}g(x_i)$ for $x_i\ge x_{j+1}$ and $i\le n+1$, the trapezoidal approximation yields 
\begin{align*}\frac{g(x_{j+1})-g(x_j)}{x_{j+1}^{\alpha-1}-x_j^{\alpha-1}}&\left(\int_{x_j}^{x_{j+1}}\frac{(x_i-y)^{\alpha-1}}{\Gamma(\alpha)}(y^{\alpha-1}-x_j^{\alpha-1})\,dy\right.\\&\;-\left.\frac{h}2 \frac{(x_i-x_{j+1})^{\alpha-1}(x_{j+1}^{\alpha-1}-x_j^{\alpha-1})}{\Gamma(\alpha)}\right)\\
=&\frac{g(x_{j+1})-g(x_j)}{x_{j+1}^{\alpha-1}-x_j^{\alpha-1}} h^3\frac{f''(\xi_j)}{12},
\end{align*}
and 
\begin{align*}
    g(x_j)&\left(\int_{x_j}^{x_{j+1}}\frac{(x_i-y)^{\alpha-1}}{\Gamma(\alpha)}\,dy-\frac{h}2\frac{(x_i-x_{j+1})^{\alpha-1}+(x_i-x_j)^{\alpha-1}}{\Gamma(\alpha)}\right)\\
    =&-h^3g(x_j)\frac{(x_i-\xi_j)^{\alpha-3}}{12\Gamma(\alpha-2)}
\end{align*}
for some $\xi_j\in [x_j,x_{j+1}]$ and $f(y)=\frac{(x_i-y)^{\alpha-1}}{\Gamma(\alpha)}(y^{\alpha-1}-x_j^{\alpha-1})$.

Note that for $0<j<i-1$,
\[h^3(x_i-\xi_j)^{\alpha-3}\le h^{\alpha}(i-j-1)^{\alpha-3},\]
\begin{align*}h^3(x_i-\xi_{j})^{\alpha-2}\xi_j^{\alpha-2}\le &\,h^3 (x_i-x_{j+1})^{\alpha-2}x_j^{\alpha-2}\\
=&\,h^2(hi)^{2\alpha-3}\left(\frac1i \left(1-\frac{j+1}i\right)^{\alpha-2}\left(\frac{j}i\right)^{\alpha-2}\right)\\
\le &\, \max\{h^{2\alpha-1},h^2\}\left(\frac1i \left(1-\frac{j+1}i\right)^{\alpha-2}\left(\frac{j}i\right)^{\alpha-2}\right),
\end{align*}
and
\[h^3(x_i-\xi_j)^{\alpha-1}\xi^{\alpha-3}\le h^{\alpha}j^{\alpha-3}.\]
As \[\sum_{j=1}^{i-2}(i-j-1)^{\alpha-3}\le \sum_{j=1}^\infty j^{\alpha-3}=\zeta(3-\alpha) < \infty\] and \[\sum_{j=1}^{i-2}\frac1i \left(1-\frac{j+1}i\right)^{\alpha-2}\left(\frac{j}i\right)^{\alpha-2}\to \int_0^1(1-y)^{\alpha-2}y^{\alpha-2}\,dy,\]
the sums are bounded independent of $i$. Since for $j=0$ or $j=i-1$ the error is also of order $h^\alpha$, the theorem follows.
\end{proof}

The next lemma sets the scene for the regularity needed of the initial condition to facilitate higher order convergence of our finite difference scheme. We require regularity of order $\alpha+1$ except for the two terms containing $x^{\alpha-1}$ and $x^{2\alpha-1}$ usually found as part of elements in $D((A^\alpha) ^\infty)$ ensuring that if $f\in D((A^\alpha)^2)$ then $f$ will have the required regularity.   

\begin{lem}
Let \[f(x)=ax^{\alpha-1}+bx^{2\alpha-1}+\int_0^x\frac{(x-s)^{\alpha}}{\Gamma(\alpha+1)}g(s)\,ds\] for some $g\in C(0,1]$ and assume $f(1)=0$. Then $A^\alpha_h\Pi_n f$ satisfies the conditions of Proposition \ref{inverseEst} with \[C_{A^\alpha_h\Pi_n f}\le C(|b|+\|g\|).\] Furthermore, 
\[\|A^\alpha_h\Pi_nf\|\le C(|b|+\|\int g\|).\]
\end{lem}
\begin{proof}
By design, $A^\alpha_h\Pi_n x^{\alpha-1}=0$. Consider $g_\alpha$ of equation \eqref{gbeta}. Then
\[A^\alpha_h\Pi_n \frac{x^{2\alpha-1}}{\Gamma(2\alpha)}(x_i)= \frac{x_i^{\alpha-1}}{\Gamma(\alpha)}+h^{\alpha}g_\alpha\left(\frac{x_i}h\right)/h=\frac{x_i^{\alpha-1}}{\Gamma(\alpha)}+h^{\alpha-1}g_\alpha\left(i\right).\] As $g_\alpha\in L^1$ and $g_\alpha$ is $\alpha$ times differentiable, its derivative is also a bounded continuous $L^1$ function. As any function with a bounded continuous derivative satisfies the conditions of Proposition \ref{inverseEst} so does  $A^\alpha_h\Pi_n x^{2\alpha-1}$ and the last term of $f$. 

As by Theorem \ref{thm:convonR} $A^\alpha_h\Pi_n f\to b\frac{\Gamma(\alpha+1)}{\Gamma(\alpha)}x^{\alpha-1}+\int_0^x g(s)\,ds$ uniformly on grid points, 
\[\|A^\alpha_h\Pi_nf\|\le C(|b|+\|\int g\|).\]
\end{proof}

\begin{thm}
Let \[f(x)=ax^{\alpha-1}+bx^{2\alpha-1}+\int_0^x\frac{(x-s)^{\alpha}}{\Gamma(\alpha+1)}g(s)\,ds\]  with $f(1)=0$ and some $g\in C(0,1]$ be the initial condition to \eqref{FPDE}. Then, for $t>h^\alpha$, there exists a constant $C>0$ such that
\[\|S(t)f-S_h(t)\Pi_nf\|\le Ch^\alpha(1+\log(t/h^\alpha))(\|f\|+|b|+\|g\|).\]
\end{thm}
\begin{proof}
First note that both semigroups are contraction semigroups. Hence
\[\|S(t)f-S(t)\Pi_nf\|\le C\|f-\Pi_nf\|\le Ch^\alpha\|A^\alpha f\|\le Ch^\alpha(|b|+\|\int g\|).\]
Next rewrite
\begin{align*}
S(t)\Pi_nf-S_{h}(t)\Pi_nf=&\,\int_0^tS(r)(A^\alpha-A^\alpha_{h})S_{h}(t-r)\Pi_nf\,dr\\
=&\,\int_0^{h^\alpha}S(r)(A^\alpha-A^\alpha_{h})S_{h}(t-r)\Pi_nf\,dr \\
&+ \int_{h^\alpha}^t A^\alpha S(r)(\left(A^\alpha_h\right)^{-1}-\left(A^\alpha\right)^{-1})A^\alpha_{h}S_{h}(t-r)\Pi_nf\,dr\\
=&\,e_1+e_2.
\end{align*}
Integration by parts then yields
\begin{align*}
    e_1=&\,\int_0^{h^\alpha}S(r)A^\alpha S_{h}(t-r)\Pi_nf\,dr -\int_0^{h^\alpha}S(r)A^\alpha_{h}S_{h}(t-r)\Pi_nf\,dr \\
    =&\,S(h^\alpha) S_h(t-h^\alpha)\Pi_nf-S_h(t)\Pi_nf,
\end{align*}
which we further decompose into
\begin{align*}
    e_1=&\,S(h^\alpha)\left[ S_h(t-h^\alpha)-S_h(t)\right]\Pi_nf+\left[S(h^\alpha)-I\right]\left(A^\alpha_h\right)^{-1}S_h(t)A^\alpha_h\Pi_nf\\
    =&\,S(h^\alpha)\left[ S_h(t-h^\alpha)-S_h(t)\right]\Pi_nf\\
    &+\left[S(h^\alpha)-I\right]\left(A^\alpha_h\right)^{-1}-\left(A^\alpha\right)^{-1})S_h(t)A^\alpha_h\Pi_nf\\
    &+\left[S(h^\alpha)-I\right]\left(A^\alpha\right)^{-1}S_h(t)A^\alpha_h\Pi_nf.
\end{align*}
Hence
\[\|e_1\|\le h^\alpha\|A^\alpha_h\Pi_nf\|+Ch^\alpha(\|A^\alpha_h\Pi_n f\|+C_{S_h(t)A^\alpha_h\Pi_nf})+h^\alpha\|A^\alpha_h\Pi_n f\|,\]
where $C_g$ is the constant of Proposition \ref{inverseEst}. As $S_h$ is a contraction, \[C_{S_h(t)A^\alpha_h\Pi_nf}\le C_{A^\alpha_h\Pi_nf}\le C(|b|+\|g\|).\]

To bound $e_2$, we use the analyticity of $S$ to get
\begin{align*}
\|e_{2}\|\leq & \int_{h^\alpha}^t \frac{C}{r}\|(\left(A^\alpha_h\right)^{-1}-\left(A^\alpha\right)^{-1})S_h(t-r)A^\alpha_h\Pi_nf\|\,dr\\
\le &\,Ch^{\alpha}\log (t/h^\alpha)(\|A^\alpha_h\Pi_nf\|+C_{A^\alpha_h\Pi_n f}).
\end{align*}
As $\|A^\alpha_h\Pi_nf\|\le C(|b|+\|\int g\|)$ and $\|\int g\|\le \|g\|$, the proof is complete.
\end{proof}

\section*{Acknowledgments}


This work is partially funded by a Marsden grant administered by the Royal Society of New Zealand. The authors would like to thank Dr Lorenzo Toniazzi and Professor Christian Lubich for many valuable discussions.


\providecommand{\bysame}{\leavevmode\hbox to3em{\hrulefill}\thinspace}
\providecommand{\MR}{\relax\ifhmode\unskip\space\fi MR }
\providecommand{\MRhref}[2]{%
  \href{http://www.ams.org/mathscinet-getitem?mr=#1}{#2}
}
\providecommand{\href}[2]{#2}

\end{document}